\documentclass[12pt,a4paper]{article}

\usepackage{amsmath,amssymb}
\usepackage{ascmac}
\usepackage{amsthm}
\usepackage{graphicx}
\usepackage[top=35truemm,bottom=35truemm,
left=30truemm,right=30truemm]{geometry}
\usepackage[figurename=Fig.]{caption}

\usepackage{cite}

\title{Difference of facial achromatic numbers between two triangular embeddings of a graph}
\author{Kengo Enami \footnote{Department of Computer and Information Science, Faculty of Science and Technology, Seikei University, 3-3-1 Kichijoji-Kitamachi, Musashino-shi, Tokyo, 180-8633, Japan. E-mail:~\texttt{enamikengo@gmail.com}}
\and Yumiko Ohno \footnote{Research initiatives and promotion organization, Yokohama National University, 79-5, Tokiwadai, Hodogaya-ku, Yokohama, Kanagawa 240-8501, Japan. E-mail:~\texttt{ohno-yumiko-hp@ynu.ac.jp}}
}
\date{}

\newtheorem{theorem}{Theorem}
\newtheorem{lemma}[theorem]{Lemma}

\newtheorem{claim}[theorem]{Claim}

\begin{document}
\maketitle
%\footnote[0]{2010 {\em Mathematics Subject Classification}: 05C10, 05C12.}

\begin{abstract}
A facial $3$-complete $k$-coloring of a triangulation $G$ on a surface is a vertex $k$-coloring such that every triple of $k$-colors appears on the boundary of some face of $G$.
The facial $3$-achromatic number $\psi_3(G)$ of $G$ is the maximum integer $k$ such that $G$ has a facial $3$-complete $k$-coloring.
This notion is an expansion of the complete coloring,
that is, a proper vertex coloring of a graph such that every pair of colors appears on the ends of some edge.

For two triangulations $G$ and $G'$ on a surface,
$\psi_3(G)$ may not be equal to $\psi_3(G')$
even if $G$ is isomorphic to $G'$ as graphs.
Hence, it would be interesting to see how large the difference between $\psi_3(G)$ and $\psi_3(G')$ can be.
We shall show that the upper bound for such difference in terms of the genus of the surface.
\end{abstract}

\section{Introduction}
In this paper, we consider finite and undirected graph.
A graph is called \emph{simple} if it has no loops and multiple edges.
We mainly focus on simple graphs unless we particularly mention it.
An \emph{embedding} of a graph $G$ on a surface $\mathbb{F}$ is a drawing of $G$ on $\mathbb{F}$ with no pair of crossing edges.
Technically, we regard an embedding as injective continuous map $f:G\to \mathbb{F}$,
where $G$ is regarded as a one-dimensional topological space.
We sometime consider that $G$ is already mapped on a surface and denote its image by $G$ itself to simplify the notation,
while if we deal with two or more embeddings of $G$ on a surface, 
we denote them by $f_1(G),f_2(G), \ldots$ to distinguish them.

The \emph{faces} of a graph $G$ embedded on a surface $\mathbb{F}$ are the connected components of the open set $\mathbb{F}-G$.
We denote by $V(F)$ the set of vertices in the boundary of a face $F$ of $G$,
and by $\mathcal{F}(G)$ the set of faces of $G$.
A \emph{triangulation} on a surface $\mathbb{F}$ is an embedding of a graph on $\mathbb{F}$ so that each face is bounded by a $3$-cycle.
A graph $G$ is said to have a triangulation on a surface,
if $G$ is embeddable on the surface as a triangulation.

A (vertex) \emph{$k$-coloring} of a graph $G$ is a map $c:V(G)\to \{1,2, \ldots , k\}$.
A $k$-coloring $c$ of $G$ is \emph{proper} if $c(u)\neq c(v)$ whenever two vertices $u$ and $v$ are adjacent.
For a subset $S \subseteq V(G)$,
we denote by $c(S)$ the set of colors of the vertices in $S$.
Colorings of graphs embedded on surfaces with facial constraints have attracted a lot of attention.
In particular, facially-constrained colorings of plane graphs were overviewed by Czap and Jendrol' \cite{survey}.
Many facially-constrained colorings can be translated into colorings of some kind of hypergraphs, called ``face-hypergraphs".
The \emph{face-hypergraph} $\mathcal{H}(G)$ of a graph $G$ embedded on a surface is the hypergraph with vertex-set $V(G)$ and edge-set $\{ V(F) : F\in \mathcal{F} (G)\}$, whose concept was introduced in \cite{KR}.

A \emph{complete $k$-coloring} of a graph $G$ is a proper $k$-coloring such that
each pair of $k$-colors appears on at least one edge of $G$.
The \emph{achromatic number} of $G$ is the maximum integer $k$ such that
$G$ has a complete $k$-coloring.
This notion was introduced by Harary and Hedetniemi \cite{HH},
and has been extensively studied (see \cite{HM} for its survey).
Recently, Matsumoto and the second author \cite{MO} introduced a new facially-constrained coloring, called  the ``facial complete coloring'',
which is an expansion of the complete coloring.
A $k$-coloring, which is not necessarily proper, of a graph $G$ embedded on a surface is \emph{facially $t$-complete}
if for any $t$-element subset $X$ of the $k$ colors,
there is a face $F$ of $G$ such that $X\subseteq c(V(F))$.
The maximum integer $k$ such that $G$ has a facial $t$-complete $k$-coloring is the \emph{facial $t$-achromatic number} of $G$, denoted by $\psi_t(G)$.
It seems to be natural to consider facial $t$-complete colorings for graphs embedded on a surface so that each face is bounded by a cycle of length $t$.

%Matsumoto and the second author \cite{MO}  also extended the notion of the complete coloring to hypergraphs:
%a \emph{$t$-complete $k$-coloring} of a hypergraph $H$ is a $k$-coloring of $H$ such that every $t$-tuple of $k$-colors appear on at least one edge of $H$  (a bit different definitions of a complete coloring to hypergraphs can be found in \cite{DLR, JO}).
%Then, a facial $t$-complete coloring of a graph embedded on a surface corresponds to a facial $t$-complete coloring of its face-hypergraph.

We should notice that the facial $t$-achromatic number of an embedded graph depends on the embedding of the graph in general.
That is, if a graph $G$ has two distinct embeddings $f_1(G)$ and $f_2(G)$ on $\mathbb{F}$,
then $\psi_t(f_1(G))$ may not be equal to $\psi_t(f_2(G))$.
Hence, it would be interesting to see how large the difference between $\psi_t(f_1(G))$ and $\psi_t(f_2(G))$ can be.
In this paper, we focus on facial $3$-complete colorings of triangulations on a surface from this point of view,
and show the upper bound for such difference as follows.

%Let $G$ be the graph obtained from $K_{2,4}$ by replacing two paths joining the vertices of degree $4$ with two paths of length $m+1$, respectively.
%Then $G$ has two embeddings $f_1(G)$ and $f_2(G)$ on the sphere
%such that $f_1(G)$ has a face bounded by the cycle of length $2m$,
%and every face of $f_2(G)$ is bounded by a cycle of length $m+2$.

%In this example, we can show that $\psi_t(f_1(G))=2m$ and $\psi_t(f_2(G))=m+2$
%(the proof is easy and hence we omit it).
%This implies that the difference of the facial $t$-achromatic numbers between two embeddings of a graph on the sphere can be arbitrarily large.
%It would also be interesting to see how large the difference of the facial $t$-achromatic numbers between two embeddings of a graph on a surface can be
%if all faces of the embeddings are bounded by cycles of length $t$.
%Actually, in this paper, we focus on facial $3$-complete colorings of triangulations on a surface,
%and show the upper bound for such difference as follows.

%Throughout this paper, we call both of the orientable genus of an orientable surface and the non-orientable genus of a non-orientable surface \emph{genera} of surfaces.
%That is, for a surface $\mathbb{F}$ of genus $g$, we have $\chi(\mathbb{F})=2-2g$,
%where $\chi(\mathbb{F})$ is the Euler characteristic of $\mathbb{F}$,
%if $\mathbb{F}$ is orientable, and we have $\chi(\mathbb{F})=2-g$ otherwise.

\begin{theorem}\label{upper}
Let $G$ be a graph which has two triangulations $f_1(G)$ and $f_2(G)$ on a surface $\mathbb{F}$,
and let $g$ be the Euler genus of $\mathbb{F}$.
If $\mathbb{F}$ is orientable, then
$$ | \psi_3(f_1(G))-\psi_3(f_2(G)) | \leq 
\begin{cases}
9g/2 & (g\leq 1)\\
27g/2-27 & (otherwise).
\end{cases}
$$
If $\mathbb{F}$ is non-orientable, then
$$ | \psi_3(f_1(G))-\psi_3(f_2(G)) | \leq 
\begin{cases}
3g & (g=1)\\
21g-27 & (otherwise).
\end{cases}
$$
\end{theorem}

Note that we can easily construct a triangulation on each surface so that its facial $3$-achromatic number is an arbitrarily large,
while Theorem~\ref{upper} implies that the difference of the 
facial $3$-achromatic numbers between two triangulations $f_1(G)$ and $f_2(G)$ on a given surface,
which is obtained from the same graph $G$,
can be bounded by a constant.

On the other hand, the upper bounds in Theorem~\ref{upper} do not seem to be sharp.
Unfortunately, we have no construction of a graph which has two triangulations on a surface whose facial $3$-achromatic numbers differ.
So one may suspect that $\psi_t(f_1(G))=\psi_t(f_2(G))$
whenever a graph $G$ has two triangulations $f_1(G)$ and $f_2(G)$ on a surface.
However, we do not believe that.
Actually, we shall show in Section \ref{multi}, the non-simple graphs having two triangulations on a surface whose facial $3$-achromatic numbers differ (the definition of the facial complete coloring can be extended to non-simple graphs naturally).
Hence, we hope that there exist such graphs for simple graphs.

We introduce some useful lemmas in Section \ref{pre}
to prove Theorem \ref{upper} in Section \ref{proof}.
Before these sections, we would like to introduce some related results dealing with other facially-constrained colorings.
For not only the facial complete coloring but also other facially-constrained colorings,
the possibility of such a coloring depends on the embedding in general.
We survey some results from this point of view in the next section.

\section{Related results}\label{related}

A \emph{rainbow coloring} (or a \emph{cyclic coloring})
is a coloring of a graph $G$ embedded on a surface
so that each face is \emph{rainbow},
that is, any two distinct vertices on its boundary have disjoint colors.
The minimum integer $n$ such that $G$ has a rainbow $n$-coloring is the \emph{rainbowness} of $G$, denoted by $\text{rb}(G)$.
An \emph{antirainbow coloring} (or a \emph{valid coloring}) is a coloring of a graph $G$ embedded on a surface
so that no face is rainbow.
The maximum integer $n$ such that $G$ has a surjective antirainbow $n$-coloring is the \emph{antirainbowness} of $G$, denoted by $\text{arb}(G)$.

Let $G$ be the graph consisting of $m\geq 3$ cycles of length $3$ with one common vertex,
which has two embeddings $f_1(G)$ and $f_2(G)$ on the sphere as shown in Fig.~\ref{fan}.
Then $G$ has $2m+1$ vertices.

\begin{figure}[htbp]
	\begin{center}
		\includegraphics[width=90mm]{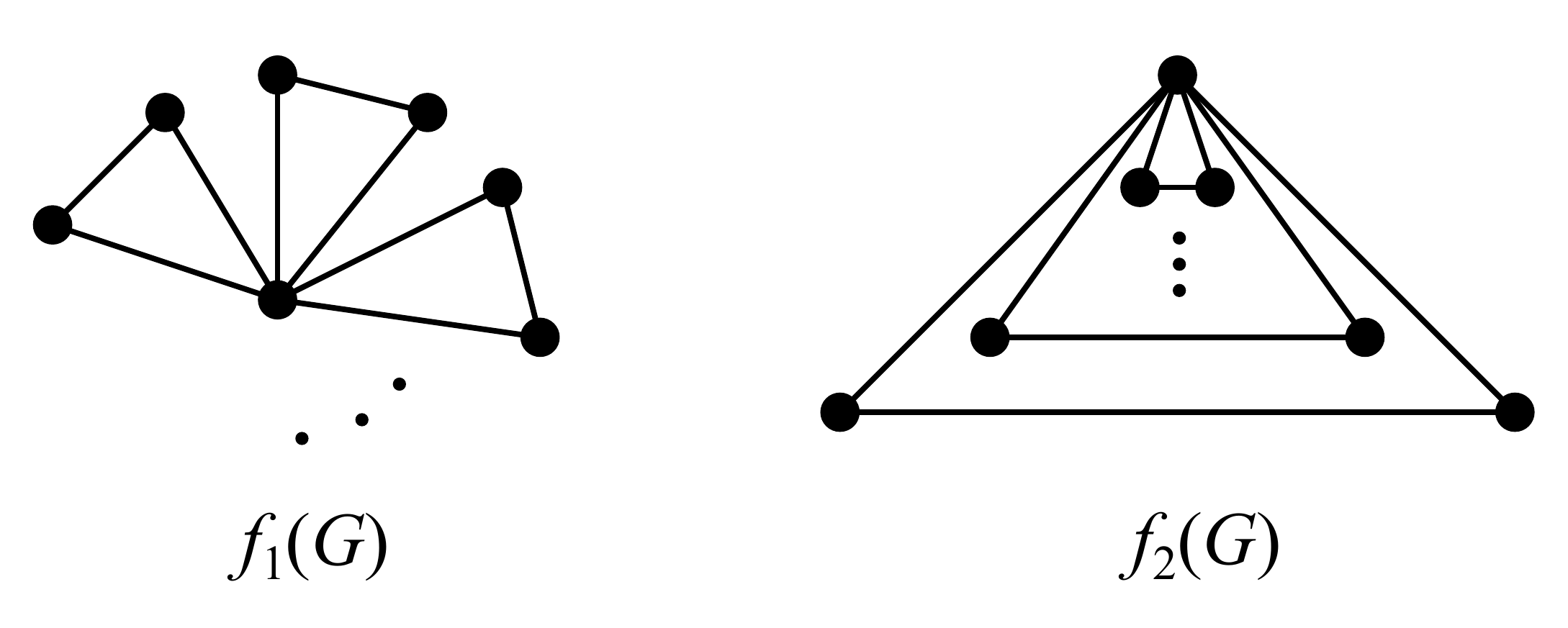}
	\end{center}
	\caption{Two embeddings of $G$ on the sphere.}
	\label{fan}
\end{figure}

As there is a face incident with all vertices in $f_1(G)$, we have $\text{rb}(f_1(G))=|V(G)|=2m+1$.
It is also easy to see that $\text{rb}(f_2(G))=5$.
Hence, the difference between $\text{rb}(f_1(G))$ and $\text{rb}(f_2(G))$ is $2m-4$.
It implies that the rainbowness of a graph embedded on a surface depends on the embedding.
Moreover, such a difference can be arbitrarily large.
On the other hand, it is easy to see that ${\rm rb}(G)=\chi(G)$ for every triangulation on a surface, where $\chi(G)$ is the chromatic number of $G$.
This implies that the rainbowness of a triangulation does not depend on the embedding.

Ramamurthi and West \cite{RW} observed that for the above two embeddings $f_1(G)$ and $f_2(G)$ of $G$ on the sphere,
$\text{arb}(f_1(G))=m+1$ and $\text{arb}(f_2(G))=\lceil 3m/2 \rceil$.
Then the difference of these antirainbownesses is $\lceil m/2 -1 \rceil$,
and hence the antirainbowness of a graph embedded on a surface also depends on the embedding.
Ramamurthi and West \cite{RW} conjectured this difference is the maximum difference for two embeddings of a graph on the sphere,
that is, for every planar graph $G$ of order $n$, there is no pair of embeddings of $G$ on the sphere whose antirainbownesses  differ from at least $\lfloor (n-2)/4 \rfloor$.

Arocha, Bracho and Neumann-Lara \cite{ABN} studied the antirainbow $3$-colorability of triangulations obtained from complete graphs, which they called the \emph{tightness}.
They proved that the complete graph of order $30$ has both of a tight triangulation and an untight one on the same surface.
This implies that the antirainbowness of triangulations depends on the embedding.
As the generalization of their work,
Negami \cite{Negami} introduced the \emph{looseness} of a triangulation $G$ on a surface,
which corresponds to $\text{arb}(G)+2$.
He proved that
for any graph having two triangulations $f_1(G)$ and $f_2(G)$ on a surface $\mathbb{F}$ of Euler genus $g$,
$ \vert \text{arb}(f_1(G))-\text{arb}(f_2(G)) \vert \leq 2 \lfloor g/2 \rfloor $.

A \emph{weak coloring} of a graph $G$ embedded on a surface is a coloring of $G$ such that no face is \emph{monochromatic},
that is, all vertices on its boundary have the same color.
Note that a weak coloring of an embedded graph corresponds to a proper coloring of its face-hypergraph.
The \emph{weak chromatic number} of $G$, denoted by $\chi_w(G)$, is the minimum integer $k$ such that
$G$ has a weak $k$-coloring.
K\"{u}ndgen and Ramamurthi \cite{KR} studied weak colorings of graphs embedded on surfaces from various viewpoints
and conjectured that
for each positive integer $k$, there is a graph that has two different embeddings on the same surface whose weak chromatic numbers differ by at least $k$.
Recently, the first author and Noguchi \cite{Enami} answered this conjecture affirmatively in two ways.
They first constructed two distinct embeddings of a simple graph on a surface
such that one of them has a weak $2$-coloring but the other has arbitrarily large weak chromatic number.
They second showed that there are non-simple graphs $G$ having two triangulations $f_1(G)$ and $f_2(G)$ on a surface with $\chi_w(f_1(G))=|V(G)|/2$ and $\chi_w(f_2(G))\leq |V(G)|/3$.

\section{Cycles in a triangulation}\label{pre}

To prove Theorem~\ref{upper},
we give some notations and introduce some lemmas.

Let $G$ be a graph and $H$ be a subgraph of $G$.
An edge not in $H$ but with both ends in $H$ is called \emph{chord} of $H$.
A subgraph $H$ of a graph $G$ is \emph{induced}
if $H$ has no chord.
An \emph{$H$-bridge} is a subgraph of $G$ induced by a chord of $H$, or a component of $G-V(H)$ together with all edges joining it to $H$.
In an $H$-bridge, a vertex belongs to $V(H)$ is called a \emph{vertex of attachment}.
Note that any two $H$-bridges are edge-disjoint
and meet only the common vertices of attachment.
(See \cite{text} for more details of $H$-bridges.)

\begin{lemma}\label{bridge}
Let $G$ be a triangulation on a surface,
and $C_1,C_2, \ldots ,C_k$ be vertex-disjoint facial cycles of $G$.
If there is no chord in the union $H=C_1\cup C_2\cup  \cdots \cup C_k$,
then there is only one $H$-bridge in $G$.
\end{lemma}

\begin{proof}
Let $C=uvw$ be a facial cycle of $G$ bounded by three vertices $u,v$ and $w$.
Suppose that $C$ is not contained in $H$.
Since $H$ consists of vertex-disjoint cycles and has no chord,
$C$ meets at most one cycle of $H$.
Suppose that $C$ meets $C_1$ at a vertex, say $u$, and $v,w\not\in V(C_i)$ for any $1\leq i \leq k$.
If $v$ and $w$ belongs to different $H$-bridges in $G$,
then the edge $vw$ joins these $H$-bridges, a contradiction.
Hence, $v$ and $w$ belongs to the same $H$-bridge in $G$.
It implies that all vertices and edges around $C_i$ belongs to one $H$-bridge in $G$.
Suppose that $C$ meets none of $C_1,C_2, \ldots ,C_k$.
Then it is clear that $u,v$ and $w$ belong to the same $H$-bridge in $G$.
Therefore, there is only one $H$-bridge in $G$.
\end{proof}

Let $G$ be a graph embedded on a surface $\mathbb{F}$.
A cycle $C$ of $G$ is \emph{contractible}
if it bounds a disk in $\mathbb{F}$,
and \emph{separating}
if it separates $\mathbb{F}$ into two parts.
We say that $C$ is \emph{$2$-sided}
if it divides its annular neighbourhood into two parts,
and is \emph{$1$-sided} otherwise.
Note that a non-separating cycle of $G$ must be non-contractible,
and if a separating cycle $C$ of $G$ is not facial
then there are at least two $C$-bridges in $G$.

\begin{lemma}\label{facial}
Let $G$ be a graph which has two triangulations $f_1(G)$ and $f_2(G)$ on a surface,
and $C$ be a $3$-cycle of $G$.
If $f_1(C)$ is facial in $f_1(G)$ but $f_2(C)$ is not facial in $f_2(G)$,
then $f_2(C)$ is non-contractible in $f_2(G)$.
\end{lemma}

\begin{proof}
Suppose to that $f_2(C)$ is contractible in $f_2(G)$.
Since $f_2(C)$ is not facial in $f_2(G)$,
it separates $f_2(G)$ into two components.
On the other hand, since $f_1(C)$ is facial in $f_1(G)$,
it follows from Lemma~\ref{bridge} that $G$ has only one $C$-bridge in $G$, a contradiction. 
\end{proof}

For two disjoint cycles $C_1$ and $C_2$ of a graph embedded on a surface $\mathbb{F}$,
cut the surface $\mathbb{F}$ along them.
When one of the component of the resulting surface
is an annulus with boundary components $C_1$ and $C_2$,
we say that $C_1$ and $C_2$ are \emph{homotopic}.

We introduce two lemmas about sets of pairwise non-homotopic cycles.
The second lemma closely follows from the proof of \cite[Proposition~3.7]{MM}, which corresponds to the first one.
However, to keep the paper self-contained, we give its proof.

\begin{lemma}[Malni\v{c} and Mohar \cite{MM}]\label{prop}
Let $G$ be a graph embedded on a surface $\mathbb{F}$,
and let $g$ be the Euler genus of $\mathbb{F}$.
Let $\Gamma$ be a set of pairwise disjoint, non-contractible and pairwise non-homotopic cycles of $G$.
If $\mathbb{F}$ is orientable, then
$$ |\Gamma| \leq 
\begin{cases}
	g/2 & (g\leq 2)\\
	3g/2 - 3 & (otherwise).
\end{cases}
$$
If $\mathbb{F}$ is non-orientable, then
$$ |\Gamma| \leq 
\begin{cases}
	g & (g\leq 1)\\
	3g - 3 & (otherwise).
\end{cases}
$$
\end{lemma}

\begin{lemma}\label{prop2}
Let $G$ be a graph embedded on a non-orientable surface $\mathbb{F}$ of Euler genus $g$.
let $\Gamma_1$ (resp.~$\Gamma_2$) be a set of pairwise disjoint, non-contractible and pairwise non-homotopic $1$-sided (resp.~$2$-sided) cycles of $G$.
Then $|\Gamma_1| \leq g$ and
$$ |\Gamma_2| \leq 
\begin{cases}
0 & (g= 1)\\
2g - 3 & (otherwise).
\end{cases}
$$
\end{lemma}

\begin{proof}
It is easy to see that this lemma holds for $g\leq 2$.
Hence, we may assume that $g\geq 3$. 
Moreover, we may assume that $\Gamma_1$ is maximal,
that is there is no $1$-sided cycle in $G$ disjoint from $\Gamma_1$.
Cutting $\mathbb{F}$ along the cycles in $\Gamma_1$,
we obtain a connected surface, denoted by $\mathbb{F}'$,
which has $|\Gamma_1|$ boundary components.
Thus, $\chi(\mathbb{F}')\leq 2-|\Gamma_1|$.
Since $\chi(\mathbb{F}')=\chi(\mathbb{F})=2-g$, 
we have $|\Gamma_1| \leq g$.

We may also assume that $\Gamma_2$ is maximal,
that is, all $2$-sided cycles in $G$ disjoint from $\Gamma_2$ is contractible or homotopic to some element of $\Gamma_2$.
Cut $\mathbb{F}$ along the cycles in $\Gamma_2$.
Then $\mathbb{F}$ is separated into some connected surfaces,
denoted by $\mathbb{F}_1, \mathbb{F}_2, \ldots , \mathbb{F}_k$.
Note that they are all compact and with non-empty boundary.
We denote by $b(\partial \mathbb{F}_i)$ the number of boundary components of $\mathbb{F}_i$ for $1\leq i \leq k$.
Since each cycle in $\Gamma_2$ gives rise to two boundary components,
we have $\sum_{i=1}^k b(\partial \mathbb{F}_i)=2|\Gamma_2|$.

Let $\mathbb{F}_1^*, \mathbb{F}_2^*, \ldots , \mathbb{F}_k^*$ be the surfaces obtained from $\mathbb{F}_1, \mathbb{F}_2, \ldots , \mathbb{F}_k$ by pasting a disk to each boundary component.
By the maximality of $\Gamma_2$,
$\mathbb{F}_i^*$ is the sphere or the projective plane for $1\leq i \leq k$.
We denote by $n_s$ and $n_p$ the numbers of the spheres and the projective planes among $\mathbb{F}_i^*$'s, respectively.
Then we have $n_p\leq g$ and $\sum_{i=1}^k \chi(\mathbb{F}_i^*)=2n_s+n_p$.

Now we shall show that if $\mathbb{F}_i^*$ is the sphere,
then $b(\partial \mathbb{F}_i)\geq 3$.
If $b(\partial \mathbb{F}_i)=1$,
then $\mathbb{F}_i$ is a closed disk, that is, the cycle bounding $\mathbb{F}_i$ is contractible in $\mathbb{F}$, a contradiction.
Suppose that $b(\partial \mathbb{F}_i)=2$.
Then $\mathbb{F}_i$ is an annulus.
If two cycles of $\Gamma_2$ corresponding to the boundary components $\mathbb{F}_i$
are the same,
then $\mathbb{F}$ must be the Klein bottle, a contradiction.
Thus, these two cycles are different from each other.
However, in this situation, they are homotopic in $\mathbb{F}$, a contradiction.
Therefore, we may assume that $b(\partial \mathbb{F}_i)\geq 3$.
It implies that $3n_s+n_p\leq 2|\Gamma_2|$.

Since $\chi(\mathbb{F})$ is equal to the sum of all $\mathbb{F}_i$'s,
we have
\begin{eqnarray*}
\chi(\mathbb{F})=\sum_{i=1}^k \chi(\mathbb{F}_i)=\sum_{i=1}^k \chi(\mathbb{F}_i^*)-\sum_{i=1}^k b(\partial \mathbb{F}_i) &=& 2n_s+n_p-2|\Gamma_2|.\\
&=& \frac{2}{3}(3n_s+n_p-2|\Gamma_2|)+\frac{1}{3}n_p-\frac{2}{3}|\Gamma_2|\\
&\leq& \frac{1}{3}g-\frac{2}{3}|\Gamma_2|.
\end{eqnarray*}

Since $\chi(\mathbb{F})=2-g$, we have $|\Gamma_2|\leq 2g-3$.
\end{proof}

\section{Proof of Theorem~\ref{upper}}\label{proof}

\begin{proof}[Proof of Theorem \ref{upper}]

Suppose that $\psi_3(f_1(G))=k$ and $\psi_3(f_2(G))<k$.
Let $c:V(G)\to \{1,2,\ldots , k\}$ be a facial $3$-complete $k$-coloring of $f_1(G)$. 
Then, every triple of $k$-colors appears in some face of $f_1(G)$.
On the other hand, some triples do not appear in the faces of $f_2(G)$.
Let $\mathcal{T}$ be a set of triples in $k$ colors
such that any triple in $\mathcal{T}$ does not appear in the faces of $f_2(G)$,
and for any pair of triples $T$ and $T'$ in $\mathcal{T}$, $T\cap T'=\emptyset$.
Moreover, we choose $\mathcal{T}$ so that $|\mathcal{T}|$ is as large as possible.
Let $T_1,T_2,\ldots ,T_m$ be the triples in $\mathcal{T}$, and so $|\mathcal{T}|=m$.
By the maximality of  $\mathcal{T}$,
we can choose $k-3m$ colors so that
every triple in these colors appear in some face of $f_2(G)$.
It implies that $f_2(G)$ has a facial $3$-complete $(\max\{3, k-3m\})$-coloring.
Then, $|\psi_3(f_1(G))-\psi_3(f_2(G))|\leq 3m$.

Let $\mathcal{C}=\{ C_1, C_2, \ldots , C_m \}$ be a set of facial cycles in $f_1(G)$
such that $c(V(C_i))=T_i$ for $1\leq i \leq m$.
Since every $C_i$ is not facial in $f_2(G)$,
it follows from Lemma~\ref{facial} that every $C_i$ is non-contractible in $f_2(G)$. 

\begin{claim}\label{cl}
There are at most three pairwise homotopic cycles of $\mathcal{C}$ in $f_2(G)$.
\end{claim}

\begin{proof}
Suppose that $C_1,C_2,C_3$ and $C_4$ are pairwise homotopic in $f_2(G)$,
and appear on the annulus bounded by $C_1$ and $C_4$ in this order.
Thus, the union $C_2 \cup C_4$ separates $C_1$ from $C_3$,
and hence there are no chords of $C_1 \cup C_3$.
Similarly, $C_1 \cup C_3$ also separates $C_2$ from $C_4$.
It implies that there are at least two $C_1 \cup C_3$-bridges in $G$.
On the other hand, since both of $C_1$ and $C_3$ are facial in $f_1(G)$ and $C_1 \cup C_3$ has no chord,
it follows from Lemma~\ref{bridge} that there is only one $C_1 \cup C_3$-bridge in $G$,
a contradiction.
Therefore, there are at most three pairwise homotopic cycles of $\mathcal{C}$ in $f_2(G)$.
\end{proof}

Now we shall give the upper bound for $|\mathcal{T}|=m$,
which induces the upper bound for $|\psi_3(f_1(G))-\psi_3(f_2(G))|$.
We first consider the case when the surface $\mathbb{F}$ is homeomorphic to one of the sphere, the projective plane, and the torus.
Suppose that $\mathbb{F}$ is the sphere.
All cycles in $G$ is contractible, and hence $\mathcal{C}=\emptyset$.
Actually, it follows Lemma~\ref{facial} that $f_1(G)$ and $f_2(G)$ are essentially equivalent embeddings.
(In general, Whitney \cite{Whitney} showed that every $3$-connected planar graph has essentially unique embedding in the sphere.)
Suppose that $\mathbb{F}$ is the projective plane.
There is no pair of disjoint non-contractible cycles in $f_2(G)$,
and hence $m \leq 1$.
Suppose that $\mathbb{F}$ is the torus.
All non-contractible and pairwise disjoint cycles in $G$ are pairwise homotopic. 
Then, all cycles in $\mathcal{C}$ are pairwise homotopic by Lemma~\ref{prop},
and hence it follows from Claim~\ref{cl} that $m \leq 3$.

Second, suppose that $\mathbb{F}$ is an orientable surface of genus at least two.
If $m> 9g-9$, then there are at least four pairwise homotopic cycles in $\mathcal{C}$ by Lemma~\ref{prop},
which contradicts Claim~\ref{cl}.
Hence, we have $m\leq 9g-9$.
Finally, suppose that $\mathbb{F}$ is a non-orientable surface of genus at least two.
If $m>7g-9$, then there are at least $6g-8$ $2$-sided cycles in $\mathcal{C}$,
and hence some four of them are pairwise homotopic by Lemma~\ref{prop2},
which contradicts Claim~\ref{cl}.
Therefore, in any case, the desired inequality holds.
\end{proof}

\section{Facial complete colorings of non-simple graphs}\label{multi}

In this section, we consider graphs which may have multiple edges.
We denote by $K_n$ the complete graph of order $n$,
and denote by $K_n^m$ the non-simple graph obtained from $K_n$ by replacing each edge with $m$ multiple edges.

The first author \cite{Enami} constructed two triangulations $f_1(G)$ and $f_2(G)$ obtained from the graph $G=K_{12m}^{6m-1}$ on a surface for any positive integer $m$.
The weak chromatic numbers of these triangulations differ by at least $2m$,
and hence his construction gives an affirmatively answer of K\"{u}ndgen and Ramamurthi's conjecture \cite[Conjecture 8.1]{KR} (see also Section \ref{related} in this paper).
We now show that the facial $3$-achromatic numbers of these triangulations also differ.

For details of constructions of $f_1(G)$ and $f_2(G)$, see \cite[Section~3]{Enami}.
The face-hypergraph $\mathcal{H}(f_1(G))$ of $f_1(G)$ is isomorphic to a complete $3$-uniform hypergraph.
That is, the triangulation $f_1(G)$ has exactly $\binom{|V(G)|}{3}$ faces and there is a face bounded by each triple of vertices.
(Such a triangulation is called \emph{complete}, whose notion was defined in \cite{KR}.)
Then it is easy to see that $\psi_3(f_1(G))=|V(G)|=12m$.

Let $T$ be a triangulation on a surface obtained from $K_{12m}$
(by Ringel's Map Color Theorem \cite{Ringel}, $K_{12m}$ has a triangulation on a surface).
The edge-set of $\mathcal{H}(f_2(G))$ coincides with that of $\mathcal{H}(T)$ by ignoring the multiplicity of the edge-sets. 
It implies that $\psi_3(f_2(G))=\psi_3(T)$.
Suppose that $T$ is facially $3$-complete $k$-colorable.
Then, $T$ must have at least $\binom{k}{3}$ faces,
and hence we obtain the following inequality:
\begin{eqnarray*}
|\mathcal{F}(T)|=4m(12m-1) &\geq & k(k-1)(k-2)/6\\
288m^2-24m &\geq & (k-2)^3\\
\sqrt[3]{288}\: m^{2/3} & \geq & k-2\\
7m+2 &\geq & k.
\end{eqnarray*}

Then, $\psi_3(f_2(G))\leq 7m+2$ (this bound might be loose),
and hence we have
$$\psi_3(f_1(G))-\psi_3(f_2(G))\geq 5m-2.$$
Since $G$ is isomorphic to $K_{12m}^{6m-1}$,
both of two triangulations $f_1(G)$ and $f_2(G)$ are embedded on a surface of Euler genus $(m-1)(m-2)(2m+3)/3$.
It implies that for any non-negative integer $g$,
there is a graph having two triangulations on a surface of Euler genus at least $g$,
whose facial $3$-achromatic numbers differ from $\Omega(\sqrt[3]{g})$.

%\section*{Acknowledgements}
%The first author's work was supported by JSPS KAKENHI Grant Number 19J13359.
%Many thanks also to the referees for their careful reading and helpful comments which improved the quality of this paper.

\end{document}